\documentclass{siamart171218}

\usepackage{amsmath,amsfonts,hyperref,bm}
\usepackage{amssymb}
\usepackage{xypic}
\usepackage{graphicx}
\usepackage{mathrsfs}
\usepackage{booktabs}
\usepackage{verbatim}
\usepackage{caption,afterpage}
\DeclareMathOperator*{\argmin}{arg\,min}
\DeclareMathOperator*{\argmax}{arg\,max}
\usepackage[ruled,algo2e]{algorithm2e}
\usepackage{tikz}

\providecommand{\PP}{\mathbb{P}}

\providecommand{\RR}{\mathbb{R}}

\providecommand{\TT}{\mathbb{T}}

\providecommand{\ZZ}{\mathbb{Z}}

\providecommand{\bf}{\mathbf{f}}

\providecommand{\bm}{\mathbf{m}}

\providecommand{\cG}{\mathcal{G}}

\DeclareMathOperator{\prox}{prox}
\providecommand{\abs}[1]{\lvert #1 \rvert}

\providecommand{\norm}[1]{\lVert #1 \rVert}

\begin{document}
\title{Solving Large-Scale Optimization Problems with a Convergence Rate Independent of Grid Size \thanks{Submitted to editors \today. \funding{MJ and FL were supported by NSF grant DMS-1737770 and DARPA award FA8750-18-2-0066. WL and SO were supported by DOE grant DE SC00183838 and SO was supported by STROBE NSF grant 1554564.}}}    

\author{Matt Jacobs \thanks{Corresponding/first author.  Department of Mathematics UCLA, Los Angeles, CA 90095 USA \email{majaco@math.ucla.edu} } 
\and 
Flavien L{\'e}ger\thanks{Department of Mathematics UCLA, Los Angeles, CA 90095 USA  (\email{flavien@ucla.edu}, \, \email{wcli@math.ucla.edu}, \, \email{sjo@math.ucla.edu})}
\and 
Wuchen Li \footnotemark[3]
\and
Stanley Osher \footnotemark[3]
}

\maketitle

\date{\today}

\begin{abstract}
We present a primal-dual method to solve $L^1$-type non-smooth optimization problems independently of the grid size. We apply these results to two important problems : the Rudin--Osher--Fatemi image denoising model and the $L^1$ earth mover's distance from optimal transport. Crucially, we provide analysis that determines the choice of optimal step sizes and we prove that our method converges independently of the grid size. Our approach allows us to solve these problems on grids as large as $4096\times 4096$ in a few minutes without parallelization. 
\end{abstract}


\section{Introduction}\label{sec:introduction}

 In recent years there has been an explosion of interest (\cite{goldstein_osher, chambolle_pock, nesterov_composite,nemirovski_saddle, chambolle_pock_imaging_book} and many more) in solving convex optimization problems using  first-order algorithms. The primary advantage of first-order algorithms (as compared to say Newton's method) is that one need only evaluate the proximal operator or the gradient of the functional at the current position.  As a result, the complexity of each iteration is typically linear in the total number of grid points.  This opens the door to solving extremely large problems, which would be infeasible with other methods.

However, such a viewpoint often sweeps under the rug that the convergence rate of first-order methods may depend badly on the size of the problem.   This dependence may enter through two competing factors---the distance between the minimizer and the initial point, and the stability of the descent information.  These factors are easiest to understand in the context of smooth gradient descent. Indeed, given a smooth convex functional $F$ with a unique global minimum at $u^*$, gradient descent using the inner product $(\cdot, \cdot)_{\mathcal{H}}$ has the convergence rate
\begin{equation}\label{eq:gd_rate} F(u_n)\leq F(u^*)+ 2L_{\mathcal{H}}\frac{\norm{u^*-u_0}^2_{\mathcal{H}}}{n+4} \end{equation}
where $u_n$ is the $n^{th}$ iterate, $u_0$ is the initial point and $L_{\mathcal{H}}$ is the Lipschitz constant of $\nabla_{\mathcal{H}} F$ in the norm $\norm{\cdot}_{\mathcal{H}}$  \cite{nesterov_book}.  Strengthening the inner product $(\cdot, \cdot)_{\mathcal{H}}$ decreases $L_{\mathcal{H}}$ at the expense of increasing $\norm{u^*-u_0}_{\mathcal{H}}$ (and vice versa).  In the continuum setting, if $L_{\mathcal{H}}$ or  $\norm{u^*-u_0}_{\mathcal{H}}$ is infinite then on a discrete grid the corresponding quantity will grow as the grid resolution becomes finer.  In these cases, each iteration of the first-order method is extremely efficient, but the number of required iterations depends on the problem size.  This can place a severe restriction on the size of  solvable problems.

The situation appears to be particularly dire for pathological problems where at least one of  $L_{\mathcal{H}}$ or  $\norm{u^*-u_0}_{\mathcal{H}}$ is infinite for \emph{any} choice of inner product.  In this case,  equation (\ref{eq:gd_rate}) would suggest that it is not possible to obtain a convergence rate independent of the grid size.   Our goal in this paper is to show that this is in fact not the case---even for pathological problems  the convergence rates of first-order methods can be made independent of the problem size.  

Our approach is inspired by a more powerful convergence rate estimate given by Nesterov in \cite{nesterov_book}.  A careful examination of his proof of accelerated gradient descent reveals that for smooth convex $F$ one has
\begin{equation}\label{eq:nesterov_infconv} F(u_n)\leq \min_{u} \bigg[ F(u)+4 L_{\mathcal{H}}\frac{\norm{u-u_0}^2_{\mathcal{H}}}{(n+2)^2} \bigg].\end{equation}
This estimate gives far more flexibility, as we can attempt to approximate the minimizer $u^*$ with a sequence $(u_k)$ where each $u_k$ satisfies $\norm{u_k-u_0}_{\mathcal{H}}=R_k<\infty$.  If we then let $\delta_k=F(u_k)-F(u^*)$ we see that
\begin{equation}\label{eq:approx_trick} F(u_n)-F(u^*)\leq \delta_k+4L_{\mathcal{H}}\frac{R^2_k}{(n+2)^2} .\end{equation}
As long as $\delta_k\to 0$, we can choose $k$ and $n$ so that the right hand side of (\ref{eq:approx_trick}) is as small as desired.  This perspective makes it clear that $L_{\mathcal{H}}<\infty$ should be prioritized over $\norm{u^*-u_0}_{\mathcal{H}}<\infty$ when choosing an inner product. More importantly, we can see that the convergence rate can be made independent of the problem size.

In this paper, we are interested in $L^1$-type problems where the functional $F$ is not smooth.  As such, we cannot use Nesterov's method.  Instead, we work with a modified version of Chambolle and Pock's primal-dual hybrid gradient algorithm (PDHG)~\cite{chambolle_pock}, which we call G-prox PDHG (we pause here to note that through various reductions G-prox PDHG can be shown to be equivalent to the well-known Douglas--Rachford splitting (DRS) algorithm). Primal-dual algorithms convert minimization problems 
\begin{equation}\label{eq:general_functional} F(u)=f(Ku)+g(u) \end{equation}
into saddle point problems 
\begin{equation}\label{eq:saddle1} \mathcal{L}(u, p)=(Ku, p)_{\mathcal{Z}} +g(u)- f^*(p) \end{equation}
where $f$ and $g$ are convex functions, $K:\mathcal{H}\to\mathcal{Z}$  is a linear map between Hilbert spaces, and $f^*$ is the convex dual of $f$. Morally, the original PDHG algorithm searches for a saddle-point by following the trajectory of the coupled system of equations
\begin{equation*}\label{eq:pdhg_system}
\begin{cases}
\partial_tu\in -K^Tp-\partial g(u)\\
\partial_tp\in Ku-\partial f^*(p)
\end{cases}
\end{equation*}
 where $\partial g$ and $\partial f^*$ are the subdifferentials of their respective functions.  Each step of the algorithm does the $u$ ``gradient descent'' of $\mathcal{L}(u,p)$ with step size $\tau$ and the $p$ ``gradient ascent'' of $\mathcal{L}(u,p)$ with step size $\sigma$.  In order for the scheme to be computationally feasible, the $u$ and $p$ updates are decoupled.
As a result, the scheme is stable if and only if $\tau\sigma <\frac{1}{L^2_{\mathcal{H}}}$ where $L_{\mathcal{H}}$ is the operator norm of $K$.  

 G-prox PDHG follows the trajectory of the modified system of equations 
\begin{equation*}\label{eq:pdhg_system_modified}
\begin{cases}
\partial_tu\in  -(K^TK)^{-1} \big(K^Tp-\partial g(u) \big)\\
\partial_tp\in Ku-\partial f^*(p) .
\end{cases}.
\end{equation*}
The key difference here is that the $u$ evolution equation is preconditioned by the operator $(K^TK)^{-1}$.   As a result, this scheme has a much more satisfactory stability condition 
\[
\tau\sigma<1.
\]
If $K$ is an unbounded operator, say $K=\nabla$, the step sizes of PDHG must depend on the grid resolution.   On the other hand, the step sizes of G-prox PDHG are clearly independent of the grid size.   As one might expect from our exposition above, we must pay for the increased stability by increasing the distance between the solution $u^*$ and the initial point $u_0$.  Indeed this is the case, the convergence rate will now depend on $\norm{K(u^*-u_0)}_{\mathcal{Z}}$ as opposed to $\norm{u^*-u_0}_{\mathcal{H}}$ for the original PDHG. However, this trade-off is worth it. Our main result is a Nesterov-type estimate (in the spirit of equation \eqref{eq:nesterov_infconv}) for the convergence rate of G-prox PDHG.   This estimate shows that an approximate solution to the optimization problem can be obtained independently of the grid size even when $\norm{K(u^*-u_0)}_{\mathcal{Z}}$ is infinite. 

\begin{theorem}\label{thm:main}  Suppose that $F$ is a functional of the form
$$F(u)=f(Ku)+g(u)$$
where for all $u\in \mathcal{H}$ the maximizer $p(u)=\argmax_{p\in \mathcal{Z}} (Ku,p)_{\mathcal{Z}}-f^*(p)$ is uniformly bounded, i.e.
$$\sup_{u\in \mathcal{H}} \norm{p(u)}_{\mathcal{Z}}=C<\infty.$$
Let $u^N=\frac{1}{N}\sum_{n=1}^N u_n$ and $p^N=\frac{1}{N}\sum_{n=1}^{N}p_n$ where $u_n$ and $p_n$ are the sequence of iterates produced by G-prox PDHG starting from $u_0$ and $p_0$.  Then there is an optimal choice of step sizes $\tau$ and $\sigma$ such that after $N$ iterations

$$ F(u^N)\leq \min_{R \geq 0}\, \bigg[ \frac{RC}{N}+\min_{\norm{K (u-u_0)}_{\mathcal{Z}}\leq R} F(u)\bigg].$$
Furthermore, if $\lim_{R\to\infty}\min_{\norm{K(u-u_0)}_{\mathcal{Z}}\leq R} F(u)=\inf_{u\in\mathcal{H}} F(u)=\bar{F}$, then $\lim_{N\to\infty} F(u^N)=\bar{F}.$
\end{theorem}

Note that Theorem \ref{thm:main} does not explain how to choose the optimal step sizes $\tau$ and $\sigma$, nor does it give an explicit convergence rate in terms of $N$ only.  As it turns out, the convergence rate and the optimal choices of $\tau$ and $\sigma$ are highly dependent on the properties of the functional $F$ and the underlying space $\mathcal{H}$.   Furthermore, the optimal choices of $\tau$ and $\sigma$ may depend on the user's desired error tolerance.  For example, the optimal step sizes used to find an $\epsilon$ accurate solution may be different than the optimal step sizes used to find an $\epsilon/2$ accurate solution!  

In the face of such complication, it seems unlikely that there is an elegant or concise statement which provides the optimal convergence rate and optimal step sizes for general $F$.  Instead, we focus on two important  problems: the Rudin--Osher--Fatemi (ROF) image denoising model and the earth mover's distance (EMD) between two probability measures.  Both of these problems can be solved very efficiently with our method, as the matrix inversion $(K^TK)^{-1}$ can be carried out in log-linear time using the Fast Fourier Transform (FFT).  For both of these problems, we provide a principled strategy for choosing approximately optimal step sizes $\tau$ and $\sigma$, and give an explicit upper bound for the convergence rate in terms of the number of iterations $N$.  To the best of our knowledge, this paper provides the first proof that these problems can be solved with a convergence rate independent of the grid discretization.  In addition, our arguments give a blueprint for extending the convergence rate results to other functionals of interest.

The rest of the paper is structured as follows.  We conclude the introduction with a summary of our contributions.  
Next, in Section \ref{sec:background}, we provide background on convex optimization and introduce further notation.  In Section \ref{sec:main_results}, we prove the main results of the paper.  In Section \ref{sec:numerical}, we perform various numerical experiments that highlight the need for our rigorous theoretical analysis. Finally, in Section \ref{sec:discussion}, we conclude the paper with a brief discussion.

\subsection*{Contributions}

The following is a summary of the present paper's contributions:
\vspace{-0.1cm}

\begin{itemize}

\item We conclusively demonstrate that the ROF problem and the $L^1$ EMD problem can be solved with a convergence rate independent of the grid discretization. Furthermore, our arguments apply to a general class of $L^1$ functionals. Surprisingly, these results are non-trivial and require a detailed analysis.  
\item Crucially, our analysis provides the optimal step sizes for splitting algorithms.  As a result, we are able to solve these problems orders of magnitude faster than previous works. We can solve problems on grids of size $2048\times 2048$ in less than 2 minutes and grids of size $4096\times 4096$ in less than 10 minutes without parallelization.  

\end{itemize}

\section{Background and Notation}
\label{sec:background}

In this paper we will be interested in convex functionals $F:X\to\RR\cup\{+\infty\}$ where $X$ is a convex subset of the space of functions $\{u:[0,1]^d:\to\RR\}$.  In order to minimize $F$, one typically needs to add additional structure in the form of a distance.  The distance is used to control how far the scheme is allowed to move using only local information about $F$.   In principle, these distances can be very general \cite{chambolle_pock_improved}.   Here, we will focus on the case where the distance is induced by an inner product $(\cdot, \cdot)_{\mathcal{H}}$.    Thus, it will be useful for us to assume that $F$ is defined on some Hilbert space $\mathcal{H}$ with the inner product $(\cdot, \cdot)_{\mathcal{H}}$.    

Typically, there is an enormous amount of freedom in choosing the Hilbert space $\mathcal{H}$  ( it is usually easy to extend $F$ to a larger space or restrict it to a smaller space).   In the introduction, we alluded to the fact that there may not be a single space Hilbert space $\mathcal{H}$ which is the ``natural'' or ``correct'' choice.  Ideally, one should choose a Hilbert space where $F$ is both Lipschitz continuous and coercive in the norm $\norm{\cdot }_{\mathcal{H}}$. This is enough to imply that $F$ has a minimizer $u^*\in \mathcal{H}$ \cite{zeidler_optimization}, and that $F$ behaves stably in local neighborhoods.   However, for many interesting functionals, no such ``natural'' Hilbert space exists.   For example, there is no Hilbert space where the total variation functional is both Lipschitz continuous and coercive.  Thus,  when choosing an inner product we must be aware of the trade-offs that such a choice entails.   

 For non-smooth functionals, we cannot use gradient descent to search for a minimizer.  Instead, we turn to the proximal operator of $F$
\begin{equation} \prox_{\tau}(F,u)=\argmin_{u'\in\mathcal{H}} F(u')+\frac{1}{2\tau}\norm{u'-u}_{\mathcal{H}}^2\,,\end{equation}
which is well-defined when $F$ is merely lower-semicontinuous and bounded below on $\mathcal{H}$ \cite{zeidler_optimization}.  Roughly speaking, the proximal operator searches for the smallest value of $F$ in a neighborhood of the current point $u$.    When $F$ is smooth, we know that $\prox_{\tau}(F,u)\approx u-\tau \nabla_{\mathcal{H}} F(u)$ for small $\tau$.  Thus, the proximal operator generalizes the notion of gradient descent to non-smooth functionals.   

Unfortunately, computing the proximal operator of a non-trivial functional $F$ is extremely difficult.  Indeed, one should expect that computing $ \prox_{\tau}(F,u)$ is as difficult as finding a minimizer of $F$.  On the other hand, a large class of non-trivial functionals $F$ can be written as a sum
\begin{equation}\label{eq:sum} F(u)=f(Ku)+g(u)\end{equation}
  where the proximal operators of $f(u)$ and $g(u)$ can be computed easily. This leads to a large class of closely related algorithms (Douglas--Rachford splitting, augmented Lagrangian, alternating direction of multipliers method, split Bregman, PDHG) \cite{drs, lions_mercier_79,glowinski_1975, gabay_mercier, goldstein_osher}, which minimize $F$ by appropriately combining the proximal operators of $f$ and $g$. For the remainder of this paper we shall focus on the PDHG algorithm.
  
\subsection{PDHG}\label{ssec:pdhg}
PDHG converts the minimization problem
\begin{equation}\label{eq:sum1} F(u)=f(Ku)+g(u)\end{equation}
 into the primal-dual saddle point problem 
\begin{equation}\label{eq:saddle} \mathcal{L}(u,p)=(Ku, p)_{\mathcal{Z}}+g(u)-f^*(p).\end{equation}
 $f^*$ is the convex dual of $f$, which is defined by the Legendre-Fenchel transform:
\begin{equation} f^*(p)=\sup_{v\in\mathcal{Z}} (v,p)-f(v). \end{equation}
For convex functions $f$, the Legendre-Fenchel transform is an involution  $f^{**}=f$. Therefore, $F$ can be recovered from $\mathcal{L}$ by
$$F(u)=\sup_{p\in\mathcal{Z}} \mathcal{L}(u,p).$$
If $F$ has a unique minimizer $u^*$ and $\mathcal{L}$ has a saddle point $(\hat{u},\hat{p})$ then $\hat{u}=u^*$.   Indeed,
$$F(\hat{u})=\sup_{p\in\mathcal{Z}}\mathcal{L}(\hat{u},p)=\mathcal{L}(\hat{u},\hat{p})\leq \mathcal{L}(u^*,\hat{p})\leq F(u^*).$$  Therefore, instead of directly minimizing $F$, we can achieve the same goal by searching for a saddle point of $\mathcal{L}$.  

To proceed further, we must know what it means for a point $(u,p)$ to be close to a saddle point.  A notion of closeness can be defined through the primal-dual gap:
$$\mathcal{G}(u,p)=\sup_{p'\in\mathcal{Z}}\mathcal{L}(u,p')-\inf_{u'\in \mathcal{H}}\mathcal{L}(u',p).$$
By definition, $\mathcal{G}(u,p)\geq 0$ for all $(u, p)$.  Furthermore, $\mathcal{G}(\hat{u}, \hat{p})=0$ if and only if $(\hat{u}, \hat{p})$ is a saddle point. Indeed, $\mathcal{L}$ is convex in $u$ for fixed $p$ and concave in $p$ for fixed $u$, thus the the inequalities
$$\sup_{p'\in\mathcal{H}} \mathcal{L}(\hat{u}, p')\leq \mathcal{L}(\hat{u}, \hat{p})\leq \inf_{u'\in\mathcal{H}}\mathcal{L}(u',\hat{p}) $$
encode the definition of a saddle point. In addition, the primal-dual gap controls how close one is to the minimizer of $F$. Namely,
$$\mathcal{G}(u,p)=F(u)-\inf_{u'\in \mathcal{H}}\mathcal{L}(u',p)\geq F(u)-\inf_{u'\in\mathcal{H}} F(u').$$

Now we are ready to discuss the PDHG algorithm.  PDHG searches for a saddle point of $\mathcal{L}$ as follows: 

\begin{algorithm2e}[H]
\SetAlgoLined

\begin{equation}\label{eq:pdhg_u} u_{n+1}=\argmin_{u\in \mathcal{H}} \; g(u)+(u,K^T\bar{p}_n)_{\mathcal{H}}+\frac{1}{2\tau}\norm{u-u_n}_{\mathcal{H}}^2\end{equation}

\begin{equation}\label{eq:pdhg_p} p_{n+1}=\argmax_{p\in \mathcal{Z}} -f^*(p)+(K u_{n+1},p)_{\mathcal{Z}}+\frac{1}{2\sigma}\norm{(p-p_n)}_{\mathcal{Z}}^2 \end{equation}

\begin{equation} \bar{p}_{n+1}=2p_{n+1}-p_n.\end{equation}

 \caption{PDHG}\label{alg:pdhg}
\end{algorithm2e} 
 Informally, PDHG can be understood as a discrete in time approximation to the trajectory of the saddle point flow

\begin{equation*}\label{eq:pdhg_system_again}
\begin{cases}
\partial_tu\in &-K^Tp-\partial g(u)\\
\partial_tp \in &\frac{\sigma}{\tau}\big(Ku-\partial f^* (p)\big) .
\end{cases}
\end{equation*}
The system becomes stationary at a point $(\hat{u}, \hat{p})$ satisfying $$0\in -K^T\hat{p}-\partial g(\hat{u}) \; \textrm{and} \; 0\in K\hat{u}-\partial f^*(\hat{p})$$
which is precisely the first order condition for $(\hat{u},\hat{p})$ to be a saddle point of $\mathcal{L}$. 
The PDHG updates can be written in terms of a semi-implicit scheme for the saddle point flow (\ref{eq:pdhg_system}):
$$\frac{u_{n+1}-u_n}{\tau}\in -K^T\bar{p}_n-\partial g(u_{n+1})$$
$$\frac{p_{n+1}-p_n}{\sigma}\in Ku_{n+1}-\partial f^*(p_{n+1}).$$
The $u$ update is not fully implicit, as the term $-K^T\bar{p}_n$ uses the value of $p$ from a previous iteration (note a fully implicit scheme is not computationally feasible).  Thus, the main source of instability in the PDHG algorithm is the decoupling of the $u$ and $p$ update steps.  The scheme is stable if $\tau\sigma \norm{K^TK}_{\mathcal{H}}<1$ \cite{chambolle_pock}.  However, if $K$ is an unbounded operator from $\mathcal{H}$ to $\mathcal{Z}$, there are no non-zero step sizes which produce a stable scheme.  Thus, we see that the underlying Hilbert spaces $\mathcal{H}$ and $\mathcal{Z}$ play a crucial role in the stability of the algorithm.

We conclude the background section with an important result of Chambolle and Pock which provides a convergence rate for the PDHG algorithm.   The convergence rate is given in terms of a slightly unusual object, the partial primal dual gap 
\begin{equation} \label{eq:partial_pd_gap} \cG_{R_1,R_2}(u,p)=\sup_{\norm{p'-p_0}_{\mathcal{Z}}\leq R_1} \mathcal{L}(u,p')-\inf_{\norm{u'-u_0}_{\mathcal{H}}\leq R_2} \mathcal{L}(u',p) \end{equation}
where $u_0$ and $p_0$ are the initial iterates of $u$ and $p$.  The partial primal dual gap restricts the search for maximizers $p'$ and minimizers $u'$ to balls of finite radius centered at the initial iterates.  As a result, it is possible for the partial primal dual gap to vanish at non saddle points.  However, if $\cG_{R_1,R_2}(\hat{u},\hat{p})$ vanishes and $\norm{\hat{p}-p_0}_{\mathcal{Z}}<R_1$, $\norm{\hat{u}-u_0}_{\mathcal{H}}<R_2$ then $(\hat{u}, \hat{p})$ is a saddle point \cite{chambolle_pock}

\begin{theorem}\label{thm:cp}\cite{chambolle_pock} Suppose that $K:\mathcal{H}\to\mathcal{Z}$ is a bounded operator and the step sizes $\tau$ and $\sigma$ satisfy $\tau\sigma\norm{K^TK}_{\mathcal{H}}<1$. Let $u^N=\frac{1}{N}\sum_{n=1}^N u_n$ and $p^N=\frac{1}{N}\sum_{n=1}^{N}p_n$ where $u_n$ and $p_n$ are the sequence of iterates produced by Algorithm \ref{alg:pdhg}.  After $N$ iterations the partial primal dual gap satisfies
\begin{equation}\label{eq:cp_pd_gap} \cG_{R_1,R_2}(u^N,p^N)\leq \frac{1}{2N}(\frac{R_1^2}{\tau}+\frac{R_2^2}{\sigma})  . \end{equation}
\end{theorem}

Formula (\ref{eq:cp_pd_gap}) is very interesting.   The radii $R_1$ and $R_2$ play the same role as the distance term $\norm{u-u_0}^2_{\mathcal{H}}$ in the gradient descent convergence rate formulas (\ref{eq:gd_rate}) and (\ref{eq:nesterov_infconv}).   Similarly, the step size restriction $\tau\sigma\norm{K^TK}_{\mathcal{H}}<1$ plays the same role as the Lipschitz constant $L_{\mathcal{H}}$.      Thus, we see that the convergence rate of PDHG depends on the inner products $(\cdot, \cdot)_{\mathcal{H}}$ and $(\cdot, \cdot)_{\mathcal{Z}}$ in the same way as gradient descent.    We shall see shortly that we will be able to use these features to convert Theorem \ref{thm:cp} into our main result.

\section{Main Results}\label{sec:main_results}

Let us recall that our goal in this paper is to solve optimization problems with a convergence rate independent of the grid size.  If $K:\mathcal{H}\to\mathcal{Z}$ is an unbounded operator then PDHG is not well-defined in the continuous setting.  In the discrete approximation, $K$ will be bounded but the operator norm of $K$ will grow with the grid size.  This implies that at least one of the step sizes $\tau$ or $\sigma$ must shrink to zero as the grid resolution approaches the continuous limit.  In that case, it is clear from formula (\ref{eq:cp_pd_gap}) that the convergence rate will depend on the grid size.  Thus, our immediate goal is to modify PDHG to ensure that $K$ is always a bounded operator.

Assume that the inner product $(\cdot, \cdot)_{\mathcal{Z}}$ for the variable $p$ has already been chosen.  The simplest way to ensure that $K:\mathcal{H}\to \mathcal{Z}$ will be a bounded operator is to define the inner product $(\cdot, \cdot)_{\mathcal{H}}$ by $(u, v)_{\mathcal{H}}=(Ku, Kv)_{\mathcal{Z}}$ (note we can always assume that $K$ is injective --- it is trivial to solve for and eliminate the components of $u$ which are not coupled to $p$).    This simple modification leads us to Algorithm \ref{alg:gproxpdhg}, G-prox PDHG.    

We again note that G-prox PDHG is equivalent to the DRS algorithm (c.f. Section 4.2 in \cite{chambolle_pock}).    However, we prefer this formulation as it allows us to clearly connect to PDHG and the convergence result (\ref{eq:cp_pd_gap}). G-prox PDHG modifies Algorithm \ref{alg:pdhg} by choosing a specific inner product for the $u$ update.  Thus, G-prox PDHG is a special case of Algorithm \ref{alg:pdhg} where $K$ is a bounded operator with operator norm $\norm{K^TK}_{\mathcal{H}}=1$.  As long as $\tau\sigma<1$, the convergence result, Theorem \ref{thm:cp}, applies to G-prox PDHG.

\begin{algorithm2e}[H]
\SetAlgoLined

\begin{equation}\label{eq:gprox_u} u_{n+1}=\argmin_{u\in \mathcal{H}} \; g(u)+(Ku,\bar{p}_n)_{\mathcal{Z}}+\frac{1}{2\tau}\norm{K(u-u_n)}_{\mathcal{Z}}^2\end{equation}

\begin{equation}\label{eq:gprox_p} p_{n+1}=\argmax_{p\in \mathcal{Z}} -f^*(p)+(K u_{n+1},p)_{\mathcal{Z}}+\frac{1}{2\sigma}\norm{(p-p_n)}_{\mathcal{Z}}^2 \end{equation}

\begin{equation} \bar{p}_{n+1}=2p_{n+1}-p_n.\end{equation}

 \caption{G-prox PDHG}\label{alg:gproxpdhg}
\end{algorithm2e}

Now we still need to address the choice of the Hilbert space $\mathcal{Z}$ for the dual variable $p$, and the impact of the distances $\norm{K(u-u_0)}_{\mathcal{Z}}$ and $\norm{p-p_0}_{\mathcal{Z}}$ on the convergence rate.   The highest priority is to choose $(\cdot, \cdot)_{\mathcal{Z}}$ so that the updates (\ref{eq:gprox_u}) and (\ref{eq:gprox_p}) can be computed efficiently. Indeed, if steps (\ref{eq:gprox_u}) and (\ref{eq:gprox_p})  cannot be computed in linear or log-linear time (in the number of grid points) then the entire purpose of choosing a first order method is lost.  In the remainder of this paper we shall always take $(\cdot, \cdot)_{\mathcal{Z}}$ to be the usual $L^2$ inner product, however there may be specific problems where a different choice is more appropriate.

Now we are ready to restate and prove Theorem \ref{thm:main} which shows that for certain problems the convergence rate of G-prox PDHG is independent of the problem size even when the distance $\norm{K(u^*-u_0)}_{\mathcal{Z}}$ is infinite.  

\begin{theorem} Suppose that $F$ is a functional of the form
$$F(u)=f(Ku)+g(u)$$
where for all $u\in \mathcal{H}$ the maximizer $p(u)=\argmax_{p\in \mathcal{Z}} (Ku,p)_{\mathcal{Z}}-f^*(p)$ is uniformly bounded i.e.
$$\sup_{u\in \mathcal{H}} \norm{p(u)}_{\mathcal{Z}}=C<\infty.$$
Let $u^N=\frac{1}{N}\sum_{n=1}^N u_n$ and $p^N=\frac{1}{N}\sum_{n=1}^{N}p_n$ where $u_n$ and $p_n$ are the sequence of iterates produced by G-prox PDHG starting from $u_0$ and $p_0$.  Then there is an optimal choice of step sizes $\tau$ and $\sigma$, satisfying $\sigma\tau<1$ such that after $N$ iterations

$$ F(u^N)\leq \min_{R \geq 0}\, \bigg[ \frac{RC}{N}+\min_{\norm{K (u-u_0)}_{\mathcal{Z}}\leq R} F(u)\bigg].$$
Furthermore, if $\lim_{R\to\infty}\min_{\norm{K(u-u_0)}_{\mathcal{Z}}\leq R} F(u)=\inf_{u\in\mathcal{H}} F(u)=\bar{F}$ then $\lim_{N\to\infty} F(u^N)=\bar{F}.$
\end{theorem}

\begin{proof}

From equation (\ref{eq:cp_pd_gap}) and the definition of $C$ we have $$\cG_{C,R}(u^N,p^N)=F(u^N)-\min_{\norm{K u}_{\mathcal{Z}}\leq R}[ g(u)+(Ku,p^N)-f^*(p^N)]\leq \frac{1}{2N}(\frac{R^2}{\tau}+\frac{C^2}{\sigma}).$$
Now we wish to estimate the second term on the left hand side with a quantity related to $F$.  Immediately we can see that
\begin{align*}
\min_{\norm{K (u-u_0)}_{\mathcal{Z}}\leq R}[ g(u)+(Ku,p^N)-f^*(p^N)] & \leq \min_{\norm{K (u-u_0)}_{\mathcal{Z}}\leq R}\max_{p\in \mathcal{Z}}[ g(u)+(Ku,p)-f^*(p)] \\
&=\min_{\norm{K (u-u_0)}_{\mathcal{Z}}\leq R} F(u).
\end{align*}
Putting things together we have 
$$ F(u^N)-\bar{F}\leq  \frac{1}{2N}(\frac{R^2}{\tau}+\frac{C^2}{\sigma})+ \min_{\norm{K (u-u_0)}_{\mathcal{Z}}\leq R} F(u)-\bar{F}.$$

For any fixed $R$ the best choice of the step sizes $\tau$ and  $\sigma$ is to take $\tau=\frac{R}{C}$ and $\sigma=\frac{C}{R}$, which gives  
$$ F(u^N)-\bar{F}\leq  \frac{RC}{N}+ \min_{\norm{K (u-u_0)}_{\mathcal{Z}}\leq R} F(u)-\bar{F}.$$
Since $R$ is arbitrary, we can minimize the right hand side over $R\geq 0$ to get the first result.  For the second result, it is enough to let $R=R(N)$ grow to infinity with $N$ such that $\lim_{N\to\infty} \frac{RC}{N}=0$.   
\end{proof}

By examining the above proof, we see that the rate of convergence is governed by the gap
$$\delta_F(R)=\min_{\norm{K(u-u_0)}_{\mathcal{Z}}\leq R} F(u)-\bar{F}.$$
Obtaining explicit upper bounds for $\delta_F(R)$ is of great practical importance.   The behavior of this quantity informs our choice of step sizes, and thus directly impacts the performance of the algorithm.  We cannot expect to give a general statement on the behavior of $\delta_F(R)$, as it is highly dependent on the properties of the functional $F$.  Instead, we will closely analyze two important problems, total variation denoising and the earth movers distance. For these two problems we shall provide explicit upper bounds on convergence rate of G-prox PDHG and we shall show how to choose the optimal step sizes $\sigma$ and $\tau$.

\subsection{Total Variation Denoising of Images}\label{ssec:rof}

Image processing is a source of many important large scale problems.  Simple consumer devices, such as cell phone cameras, have pixel counts in the tens of millions.  More importantly, medical images such as MRI scans may be 3-dimensional images of physical objects.  The pixel counts of 3-dimensional images grow cubically with the resolution, thus even relatively low resolution 3-d images images have enormous pixel counts.  

Digital images are defined on either 2 or 3 dimensional grids.  At each grid point, the image takes a value in $[0,1]$, which represents the brightness of the image at that location.    In what follows,  we shall consider images as discrete approximations to a function $I: R\to [0,1]$ where $R$ is a rectangle in the plane or a box in 3-space.  By rescaling the side lengths, we shall always assume that $I$ is defined on the unit cube $[0,1]^d$.   

A fundamental problem in image processing is image denoising.  The goal of image denoising is to remove pixel errors, while preserving as much of the original image information as possible.  The most important information is typically contained in the edges of objects and scenery.  Mathematically, edges correspond to sharp discontinuities in the image intensity function.  Thus, variational models for image denoising must be able to produce discontinuous solutions.

A popular model for image denoising is the Rudin--Osher--Fatemi (ROF) model  \cite{rof}
\begin{equation}\label{eq:general_rof} \textrm{F}_{\lambda}(u,I)=\norm{u}_{TV}+\frac{\lambda}{2}\norm{u-I}_{L^2}^2  \end{equation}
where $I$ is the original image to be denoised and $\norm{u}_{TV}$ is the total variation of $u$.  For smooth functions, $\norm{u}_{TV}=\norm{\nabla u}_{L^1}$.  

We shall consider the saddle point formulation:
\begin{equation}\label{eq:rof_saddle} (\nabla u, p)_{L^2} + \frac{\lambda}{2}\norm{u-I}_{L^2}^2 -\chi_{\infty}(p).  \end{equation} 
Here $\chi_{\infty}(p)$ is the convex indicator function of the $L^{\infty}$ unit ball, i.e.

\begin{equation*}\label{eq:linf_indicator}
\chi_{\infty}(p)=\begin{cases}
0&\textrm{if} \; \abs{p(x)}\leq 1 \; \textrm{for all} \, x\in [0,1]^d,\\
\infty &\textrm{otherwise}.
\end{cases}
\end{equation*}
Clearly, $p$ will always have $L^2$ norm bounded by 1, thus the ROF problem satisfies the hypotheses of Theorem \ref{thm:main} with $C=1$. 

The G-prox PDHG updates for the ROF problem have the form
\begin{equation} u_{n+1}=\big( \lambda \tau \textrm{Id} -\Delta)^{-1}\big( \lambda \tau I +\tau\nabla \cdot \bar{p}_n -\Delta u_n \big)\end{equation}
\begin{equation} p_{n+1}(x)=\frac{p_n(x) +\sigma \nabla u_{n+1}(x)}{\max(1, \abs{p_n(x)+\sigma\nabla u_{n+1}(x)})} \end{equation}
\begin{equation} \bar{p}_{n+1}=2p_{n+1}-p_n\end{equation}
where the identity matrix $\textrm{Id}$ should not be confused with the image to be denoised $I$.   Note that the $u$ update can be conveniently expressed as a matrix vector product, whereas it is more convenient to express the $p$ update in a pointwise fashion.

 Minimizers of the ROF model are functions of bounded variation (BV).  BV functions may have discontinuities along curves, thus the model will preserve the edges of $I$ for $\lambda$ sufficiently large but finite.  For example, if $I$ is the characteristic function of a disc of radius $r$ and $\lambda>\frac{2}{r}$ then the solution is the still discontinuous function $u^*(x)=(1-\frac{2}{\lambda r})I(x)+\frac{2\pi r}{\lambda(1-\pi r^2)}$.   As a result, minimizers of the ROF model are not in general elements of the Hilbert space $H^1([0,1]^d)=\{u\in L^2([0,1]^d): \nabla u \in L^2([0,1]^d) \}$.   Thus, ROF is not coercive in the $H^1$ norm and we shall have to compute the gap $\min_{\norm{\nabla (u-u_0)}_{L^2}\leq R} \textrm{ROF}_{\lambda}(u,I)-\textrm{ROF}_{\lambda}(u^*,I)$ to obtain a convergence rate for G-prox PDHG.

\begin{proposition}\label{prop:rof_R_gap}  Given an image $I$ taking values in $[0,1]$,  the decay of the ROF gap is bounded by
\begin{equation} \min_{\norm{\nabla u}_{L^2}\leq R} \textrm{ROF}_{\lambda}(u,I)-\textrm{ROF}_{\lambda}(u^*,I)\leq \frac{3\lambda d^2 }{2R^2}\norm{u^*}^2_{TV} \end{equation}
when $u_0=0$.  

\end{proposition}
\begin{proof}
We estimate the gap by constructing approximate minimizers $u_{\delta}$ of the ROF functional such that $u_{\delta}$ has finite $H^1$ norm.  Our trick is to take the solution $u^*$ and convolve it with the Gaussian approximation to the identity $\displaystyle G_{\delta}(z)=\delta^{-d} e^{-\pi(z/\delta)^2} $ (note that convolutions can be appropriately defined on $[0,1]^d$, see the appendix for details ).

If we let $u_{\delta}=G_{\delta} * u^*$ then  $u_{\delta}$ is a $C^{\infty}$ function and thus an element of $H^1$.  Adding and subtracting $u_{\delta}$ into the $L^2$ term we get
$$\textrm{ROF}_{\lambda}(u_{\delta},I)= \norm{u_{\delta}}_{TV}+\frac{\lambda}{2}\norm{u^*-I}_{L^2}^2 +\frac{\lambda}{2}\norm{u_{\delta} -u^*}_{L^2}^2 +\lambda ( u^*-I, u_{\delta}-u^*)_{L^2} .$$
By applying Jensen's inequality to $\norm{u_{\delta}}_{TV}$ and Holder's inequality to the last term, we have 
$$\textrm{ROF}_{\lambda}(u_{\delta},I)\leq \textrm{ROF}_{\lambda}(u^*,I) +\frac{\lambda}{2}\norm{u_{\delta} -u^*}_{L^2}^2 +\lambda \norm{ u^*-I}_{L^{\infty}}\, \norm{ u_{\delta}-u^*}_{L^1}.$$
The minimizer of the ROF problem satisfies a maximum principle, therefore we know that $u^*(x)\in [0,1]$ for all $x$ \cite{caselles_chambolle_novaga_discontinuity}. It only remains to estimate the decay of $\norm{u_{\delta}-u^*}^q_{L^q}$ and the growth of $\norm{\nabla u_{\delta}}_{L^2}^2$.  See the appendix for the details on these computations.   
\end{proof}

With Proposition \ref{prop:rof_R_gap} in hand, we can now give an upper bound on the convergence rate of G-prox PDHG applied to the ROF model. 
\begin{theorem}\label{thm:rof_main_result} An $\epsilon$ approximate solution to the ROF model may be obtained in at most 

$$N= \frac{d \sqrt{3\lambda}\, \norm{u^*}_{TV}}{\epsilon^{3/2}}$$
iterations of G-prox PDHG using the step sizes $\displaystyle\tau=\frac{d \sqrt{3\lambda}\, \norm{u^*}_{TV}}{\epsilon^{1/2}}$,    $\sigma=\tau^{-1}$.  
\end{theorem}

Note that the step size in Theorem \ref{thm:rof_main_result} depends on $\norm{u^*}_{TV}$, which is unknown until the problem is solved.   We can remedy this by providing a simple estimate for $\norm{ u^*}_{TV}$ in terms of $I$.  Let $I_0$ be the average value of $I$ over the domain.  By evaluating the functional at either $I$ or the constant $I_0$, we know that $\textrm{ROF}_{\lambda}(u^*,I)\leq \min\big(\norm{I}_{TV}, \; \frac{\lambda}{2}\norm{ I_0- I}^2_{L^2} \big)$.  This bound immediately implies $\norm{u^*}_{TV}\leq \min\big(\norm{I}_{TV}, \; \frac{\lambda}{2}\norm{ I_0- I}^2_{L^2} \big)$.  Thus, we have a strategy for choosing the step sizes using only quantities available at the start of computation.

Finally, we conclude this section with a convergence result for the infinite dimensional ROF problem.  
\begin{corollary}\label{thm:rof_convergence} Let $u^N=\frac{1}{N}\sum_{n=1}^N u_n$ where $u_n$ is the sequence of primal variables produced by G-prox PDHG.  Then $u^N$ converges to the minimizer $u^*$ of the ROF problem strongly in $L^2([0,1]^d)$.  
\end{corollary}
\begin{proof}
The ROF$_{\lambda}$ functional is $\lambda$-strongly convex with respect to the $L^2$ distance.  Therefore,
$$\textrm{ROF}_{\lambda}(u,I)-\textrm{ROF}_{\lambda}(u^*,I)\geq \frac{\lambda}{2}\norm{u-u^*}_{L^2}^2.$$ 
The convergence 
$$\lim_{N\to\infty} \textrm{ROF}_{\lambda}(u^N,I)-\textrm{ROF}_{\lambda}(u^*,I)=0$$
then gives the result.
\end{proof}

\subsection{Earth Mover's Distance}\label{ssec:emd}

The earth mover's distance is a statistical distance on probability measures.   Given probability measures $\rho^1$ and $\rho^0$ defined on a space $\Omega$, the earth mover's distance measures the minimal cost required to move the distribution of $\rho^0$ onto the distribution of $\rho^1$.   The cost is measured according to a predetermined function $c(x,y)$, which gives the expense of transporting a unit of mass at location $x\in\Omega$ to location $y\in\Omega$.   Nowadays, the earth mover's distance plays important roles in machine learning, image retrieval and image segmentation  \cite{Li1, Li2,PeyreCuturi2018, Solomon}.  This widespread usage is due to the fact that the earth mover's distance incorporates the geometry of the underlying space $\Omega$ (via the cost function).

We shall concentrate on the (important) special case where $\Omega=[0,1]^d$ and the cost function is the usual Euclidean norm, $c(x,y)=\abs{x-y}$.  We shall assume that the probability measures $\rho^1, \rho^0$ are elements of the dual space $C([0,1]^d)^*$. Furthermore we assume that there exists a compact set $K\subset (0,1)^d$ such that $\rho^1(K)=\rho^0(K)=1$.  In this setting, the Earth mover's distance coincides with the following convex optimization problem 

\begin{equation}\label{eq:emd} \textrm{EMD}(\rho^1,\rho^0)=\min_{\nabla \cdot m=\rho^1-\rho^0} \int_{[0,1]^d} \abs{m}\end{equation}
where $m$ is a $d$-dimensional vector valued measure satisfying $m\cdot n=0$ on the boundary and $\abs{\cdot}$ is the $2$ norm on $d$-dimensional vectors.  

Using the Hodge decomposition we can decompose $m=u+\nabla \psi$, where $u$ is a divergence free vector field and $\nabla \psi$ is a gradient field.  Now we see that $\nabla \cdot m=\Delta \psi=\rho^1-\rho^0$.   If we let $\psi$ solve the Poisson equation $\Delta \psi=\rho^1-\rho^0$ (with zero Neumann boundary conditions) we can rewrite the problem as
\begin{equation}\label{eq:emd_div_free} \textrm{EMD}(\rho^1,\rho^0)=\min_{u} F(u, \psi)=\min_{u}\int_{[0,1]^d} \abs{u+\nabla\psi } +\chi_{\nabla^{\perp}}(u)\end{equation}
where $\chi_{\nabla^{\perp}}(u)$ is the convex indicator function encoding the constraints $\nabla \cdot u=0$ and $u\cdot n=0$.  If $\rho^1, \rho^0$ are singular measures then $\psi$ solves the Poisson equation in a weak sense only.  Thus, $\psi$ does not satisfy the usual regularity properties enjoyed by solutions to the Poisson equation.  Nonetheless, the Hardy--Littlewood--Sobolev Lemma implies that $\nabla \psi \in L^r([0,1]^d)$ for any $r<\frac{d}{d-1}$\cite{grafakos_classical_2014}.  Therefore, the right hand side of (\ref{eq:emd_div_free}) is well-defined.    In general, one can find a vector valued measure $u^*$ which minimizes $F$, however the minimizer need not be unique \cite{Evans}. 

Let us briefly note that the EMD problem is closely related to the ROF model when $d=2$.  In 2-dimensions, divergence free vector fields $u$ can be written in the form $u=\nabla^{\perp} h$ where $h$ is a scalar function and $\nabla^{\perp}h=(\partial_y h, -\partial_x h)$.  In this case, the EMD distance can be written as the unconstrained minimization problem
\begin{equation}\label{eq:emd_unconstrained} \textrm{EMD}(\rho^1,\rho^0)=\min_{h} \int_{[0,1]^2} \abs{\nabla^{\perp}h+\nabla \psi}.\end{equation} 
Now we can see that the EMD problem has the same structure as the ROF model, where we are minimizing the Euclidean norm of a differential operator applied to a function.  

Returning to equation (\ref{eq:emd_div_free}), the saddle point formulation of the EMD problem has the form
\begin{equation}  \label{eq:emd_saddle} (u+\nabla \psi, p)_{L^2}+\chi_{\nabla^{\perp}}(u)-\chi_{\infty}(p)   \end{equation}
where $\chi_{\infty}$ is defined as in equation (\ref{eq:linf_indicator}).

The G-prox PDHG updates for the EMD problem have the form
\begin{equation} u_{n+1}=u_n-\tau\PP_{\nabla^{\perp}}(\bar{p}_n) \end{equation}
\begin{equation} p_{n+1}(x)=\frac{p_n(x) +\sigma\big(u_{n+1}+ \nabla \psi(x)\big)}{\max\big(1, \abs{\,p_n(x) +\sigma\big(u_{n+1}+ \nabla \psi(x)\big)}\big)} \end{equation}
\begin{equation} \bar{p}_{n+1}=2p_{n+1}-p_n\end{equation}
where $\PP_{\nabla^{\perp}}$ is the Leray projection onto divergence free vector fields
\begin{equation} \PP_{\nabla^{\perp}}(p)=p-\nabla\Delta^{-1}\nabla\cdot p\end{equation}
Again, the $u$ update can be conveniently expressed as a matrix vector product, whereas it is more convenient to express the $p$ update in a pointwise fashion.

 It is clear from the saddle point formulation that the EMD problem satisfies the hypotheses of Theorem \ref{thm:main} with $C=1$.  Since minimizers of the EMD functional are measures, we should not expect the minimizers to have finite $L^2$ norm.  Thus, to obtain a convergence rate for the EMD problem we shall need to estimate the gap 
$$\min_{\norm{u}_{L^2}\leq R} F(u,\psi)-F(u^*,\psi).$$
Estimating the gap is more difficult than the ROF problem.  The regularity of $u^*$ is dependent on the regularity of the measures $\rho^1$ and $\rho^0$, but also on more complicated geometric properties of $\rho^1$ and $\rho^0$.  We shall estimate the gap assuming only that $\int_{[0,1]^d} \abs{\rho^1-\rho^0}\leq 2$.  As a result, we will have an upper bound on the gap which is valid for any probability measures, but may be too pessimistic when $\rho^1$ and $\rho^0$ are ``nice".  

Once again, we shall turn to convolutions.  Given a minimizer $u^*$, we construct approximate minimizers via convolution with the Gaussian kernel $u_{\delta}=G_{\delta} *u^*$.  The convolution takes vector valued measures to smooth functions, thus we have $u_{\delta}\in L^2([0,1]^d)$.  Here convolutions are an especially important tool as they preserve the divergence free constraint.   

\begin{proposition}\label{prop:emd_R_gap}   For any probability measures $\rho^1$ and $\rho^0$ the EMD gap satisfies
$$\min_{\norm{u}_{L^2}\leq R} F(u,\psi)-F(u^*,\psi)\leq C_d \Big(\frac{\textrm{EMD}(\rho^1,\rho^0)}{R^2}\Big)^{\frac{1}{d-1}}\log \bigg(\frac{R^2}{\textrm{EMD}(\rho^1,\rho^0)}\bigg)$$
where $C_d$ is a constant that depends on the dimension only. 
\end{proposition}
\begin{proof}
Let $u^*$ be a minimizer of $F(u,\psi)$. Let $u_{\delta}=G_{\delta}*u^*$ and $\psi_{\delta}=G_{\delta}*\psi$.   By the triangle inequality, we have
$$F(u_{\delta},\psi)\leq F(u_{\delta}, \psi_{\delta})+\norm{\nabla \psi_{\delta} -\nabla \psi}_{L^1}.$$
Applying Jensen's inequality to $F(u_{\delta}, \psi_{\delta})$ we have

$$F(u_{\delta},\psi)\leq F(u^*, \psi)+\norm{\nabla \psi_{\delta} -\nabla \psi}_{L^1}.$$
Thus, we only need to estimate the decay of $\norm{\nabla \psi_{\delta} -\nabla \psi}_{L^1}$ and the growth of $\norm{u_{\delta}}_{L^2}$.

In the appendix we show that 

$$ \norm{\nabla \psi_{\delta} -\nabla \psi}_{L^1}\leq \delta\, \big(\abs{\log(\delta) }+1\big)C'_d \int_{[0,1]^d} \abs{\rho^1-\rho^0}.  $$
and
$$\norm{u_{\delta}}_{L^2}\leq \Big((2\delta)^{1-d}\textrm{EMD}(\rho^1,\rho^0)\Big)^{1/2}  $$
 where $C'_d$ is a constant which depends on the dimension only.  By using $\int_{[0,1]^d} \abs{\rho^1-\rho^0}\leq 2$, and assuming $\delta<1/2$ we can simplify the bound for  $ \norm{\nabla \psi_{\delta} -\nabla \psi}_{L^1}$ to
 $$ \delta\abs{\log(\delta)}C_d$$
 Putting everything together we get the result.
\end{proof}

Now we can give an upper bound on the convergence rate of G-prox PDHG applied to the EMD problem.

\begin{theorem} \label{thm:emd_convergence} Suppose that $\rho^1$ and $\rho^0$ are probability measures on $[0,1]^d$.  Then $\textrm{EMD}(\rho^1,\rho^0)$ can be computed with error at most $\epsilon$ in 
$$N=\frac{C_d}{\epsilon}\Big(\frac{\textrm{EMD}(\rho^1,\rho^0)^{1/2}}{\epsilon \log(1/\epsilon)}\Big)^{(d-1)/2}$$
iterations of G-prox PDHG with the step sizes $\tau=C_d\Big(\frac{\textrm{EMD}(\rho^1,\rho^0)^{1/2}}{\epsilon \log(1/\epsilon)}\Big)^{(d-1)/2}$ and $\sigma=\tau^{-1}$where $C_d$ is a constant depending on the dimension only.     
\end{theorem}

 Note that we do not know the value of $\textrm{EMD}(\rho^1, \rho^0)$ until the problem is solved.  This is easily dealt with, as we have the estimate $\textrm{EMD}(\rho^1, \rho^0)\leq \int_{[0,1]^d} \abs{\rho^1-\rho^0}$.   
 
We conclude this section with a convergence result.
\begin{corollary}\label{cor:emd_convergence} Let $u^N=\frac{1}{N}\sum_{n=1}^N u_n$ where $u_n$ is the sequence of primal variables produced by G-prox PDHG.  Then there exists a subsequence that weakly converges to a measure $u^*$ which is a minimizer of the \textrm{EMD} functional.   
\end{corollary}
\begin{proof}
The sequence $u^N$ has uniformly bounded total variation.  By Banach--Alaoglu, the sequence has a weak cluster point $u^*$.  We have established the convergence $\lim_{N\to\infty} F(u^N,\psi)-\inf_{u\in L^2} F(u,\psi)=0$.  The EMD functional is weakly lower semi-continuous thus,
$$F(u^*,\psi)\leq \inf_{u\in L^2} F(u,\psi).$$
\end{proof}

\section{Numerical experiments}\label{sec:numerical}

All numerical algorithms were coded in C and executed on a single 1.6 GHz core with 8GB RAM.  Inversion of the Laplace operator was performed using the fast Fourier transform (FFT).  All FFTs were calculated using the free FFTW C library.   The code used in this paper is available on GitHub \url{https://github.com/majacomajaco/G_prox_pdhg}.

In this work we do not consider parallelization, however our approach is still amenable to parallelization. The only step of our method which is non-trivial to parallelize is the computation of the FFT.  This is not an insurmountable hurdle, as many modern parallel computing platforms, such as CUDA, have built in FFT subroutines.  

For all of our experiments, given an error tolerance $\epsilon$, we run the algorithms until the condition $F(u_n)-\inf_{u\in\mathcal{H}} F(u)<\epsilon$ is satisfied for the $n^{th}$ iterate $u_n$.  If we do not know $\inf_{u\in\mathcal{H}} F(u)$ in advance, we precompute it by running G-prox PDHG for an extremely large number of iterations to obtain a ``ground truth" value.  To ensure that this ground truth is sufficiently accurate, we use the primal dual gap to check that we are within $\frac{\epsilon}{10}$ of the exact value.  Note that the ground truth value may depend on the grid discretization.  As a result, we must compute a ground truth value for each grid size that we test. 

We will run G-Prox PDHG using a constant multiple of the step sizes from our analysis with one caveat.  On a discrete grid with $M$ points, all $L^p$ norms are equivalent up to a factor depending upon $M$.  Clearly, $\norm{u}_{L^{\infty}}\leq M \norm{u}_{L^1}$and it then follows that $\norm{u}_{L^p}\leq M^{1-1/p}\norm{u}_{L^1}$.  For the ROF problem we have $\norm{\nabla_M u}_{L^2}\leq M^{1/2}\norm{u}_{TV}$, and for the EMD problem we have $\norm{u}_{L^2}\leq M^{1/2}\norm{u}_{L^1}$.   This means that the gap $\delta_F(R)$ will vanish at a finite value $R_M$, even though one must take $R\to\infty$ in the continuum.    Let $R_{\epsilon}$ be the optimal choice of $R$ for solving the continuum problem with error at most $\epsilon$. Then in the discrete case, we will choose the step size $\tau=\min(R_M, R_{\epsilon})/C$ (as opposed to the choice $\tau=R_{\epsilon}/C$).   Calculating $R_M$ exactly is difficult, however we shall give some simple upper bounds in what follows below.

\subsection{Total Variation Denoising}

For the total variation denoising problem, we will consider a simple 2-dimensional example where the image $I:[0,1]^2\to[0,1]$ is the characteristic function of a disc of radius $1/4$ centered at $(1/2, 1/2)$.  This allows us to easily create a test image with any desired resolution.  For $\lambda>8$  the continuous solution is given by $u^*(x)=(1-\frac{8}{\lambda })I(x)+\frac{8\pi }{\lambda(16-\pi)}$.  On a discrete grid with $M$ points, the minimizer $u^*_M$ is different (and needs to be calculated) but approaches the continuum solution $u^*$ as the grid size grows \cite{wang_lucier}.  For our experiments, we take $\lambda=10$ and $20$ and $\epsilon=10^{-2}$ and $10^{-3}$ as error cutoffs.  

For our step sizes we will choose $\tau=\min\big(  \frac{\sqrt{\lambda}\norm{I}_{TV}}{\epsilon^{1/2}} , \norm{\nabla I}_{L^2}  \big)$.  This choice  depends only on quantities that are easily estimated at the start of computation.  Note that $\norm{\nabla I}_{L^2}$ gives a reasonable upper bound for the discrete grid quantity $R_M=\norm{\nabla u^*_M}_{L^2}$, and we have dropped the dimensionality constants from Theorem~\ref{thm:rof_main_result}.

In Tables~\ref{table:rof_lambda_10} and \ref{table:rof_lambda_20}, we present the results of G-prox PDHG when $\lambda=10$ and $\lambda=20$ respectively.  The algorithm converges faster for $\lambda=10$, since less fidelity to the original image is required. This behavior is predicted in our convergence rate analysis, Theorem \ref{thm:rof_main_result}, where the rate depends on $\sqrt{\lambda}$. When $\lambda=10$ and $\epsilon=10^{-2}$, the value of $ \frac{\sqrt{\lambda}\norm{I}_{TV}}{\epsilon^{1/2}}$ is smaller than $\norm{\nabla I}_{L^2}$, for every grid size.  As a result, the iteration count is the same for all grid sizes.  In the other experiments,  $ \frac{\sqrt{\lambda}\norm{I}_{TV}}{\epsilon^{1/2}}$ is larger than $\norm{\nabla I}_{L^2}$ on the smaller grids, thus the algorithm converges faster on the smaller grids.

In Table \ref{table:rof_lambda_20_cts_steps}, we re-run the $\lambda=20$ experiment where we do not allow the step sizes $\tau$ to depend on the grid discretization.  In other words we take $\tau=\frac{\sqrt{\lambda}\norm{I}_{TV}}{\epsilon^{1/2}}$ for every grid size.  In this experiment, the iteration count stays relatively uniform as the grid size changes.  This demonstrates that the algorithm has a convergence rate which is truly independent of the grid size, but when $\epsilon$ is small one can get faster convergence by taking into account the grid discretization.  
 
\begin{table}
\center{\caption{G-Prox PDHG $\lambda=10$}\label{table:rof_lambda_10}}

\begin{tabular}{ c | c | c | c | c |}
\cline{2-5}
 & \multicolumn{2}{|c|}{Error $10^{-2}$} & \multicolumn{2}{|c|}{Error $10^{-3}$} \\ \hline  
\multicolumn{1}{|c|}{Grid size} &  Iterations & Time (s) &  Iterations & Time (s) \\ \hline
\multicolumn{1}{|c|}{$512 \times 512$} & 33 & 1.1 & 61 & 1.80 \\ \hline
\multicolumn{1}{|c|}{$1024 \times 1024$} & 34 & 5.0 & 89 & 10.3\\ \hline
\multicolumn{1}{|c|}{$2048 \times 2048$} & 34 & 21.2 & 124 & 58.4 \\ \hline
\multicolumn{1}{|c|}{$4096 \times 4096$} & 34  & 86.7 & 168 & 348.4\\ \hline
\end{tabular}
\end{table}

\begin{table}
\center{\caption{G-Prox PDHG $\lambda=20$}\label{table:rof_lambda_20}}

\begin{tabular}{ c | c | c | c | c |}
\cline{2-5}
 & \multicolumn{2}{|c|}{Error $10^{-2}$} & \multicolumn{2}{|c|}{Error $10^{-3}$} \\ \hline 
\multicolumn{1}{|c|}{Grid size} &  Iterations & Time (s) &  Iterations & Time (s) \\ \hline
\multicolumn{1}{|c|}{$512 \times 512$} & 51 & 1.5 & 209 & 4.5 \\ \hline
\multicolumn{1}{|c|}{$1024 \times 1024$} & 66 & 8.2 & 232 & 24.1\\ \hline
\multicolumn{1}{|c|}{$2048 \times 2048$} & 83 & 42.4 & 265 & 112.6 \\ \hline
\multicolumn{1}{|c|}{$4096 \times 4096$} & 87  & 186.8 & 308 & 586.3  \\ \hline
\end{tabular}
\end{table}

\begin{table}
\center{\caption{G-Prox PDHG $\lambda=20$ discretization-independent step sizes}\label{table:rof_lambda_20_cts_steps}}

\begin{tabular}{ c | c | c | c | c |}
\cline{2-5}
 & \multicolumn{2}{|c|}{Error $10^{-2}$} & \multicolumn{2}{|c|}{Error $10^{-3}$} \\ \hline 
\multicolumn{1}{|c|}{Grid size} &  Iterations & Time (s) &  Iterations & Time (s) \\ \hline
\multicolumn{1}{|c|}{$512 \times 512$} & 79 & 2.0 & 505 & 10.3 \\ \hline
\multicolumn{1}{|c|}{$1024 \times 1024$} & 81 & 9.5 & 412 & 44.1\\ \hline
\multicolumn{1}{|c|}{$2048 \times 2048$} & 83 & 41.2 & 396 & 172.2 \\ \hline
\multicolumn{1}{|c|}{$4096 \times 4096$} & 86  & 176.9 & 401 & 794.5  \\ \hline
\end{tabular}
\end{table}

Next, we compare G-prox PDHG to an accelerated version of PDHG (Algorithm 2 from \cite{chambolle_pock}), which we will refer to as CP2.   CP2 has two advantages over G-prox PDHG.  CP2 has extremely simple updates which do not require solving any linear systems.  As a result, each iteration of CP2 is faster than G-prox PDHG, which needs to compute an FFT to invert $(\lambda \tau \textrm{Id} -\Delta)$. Secondly, the primal dual formulation of the ROF problem
$$\mathcal{L}(u,p)=(\nabla u, p)_{L^2} + \frac{\lambda}{2}\norm{u-I}_{L^2}^2-\chi_{\infty}(p)$$
is $L^2$ strongly convex in $u$.  CP2 computes the $u$ update in the $L^2$ norm, thus the $L^2$ strong convexity can be exploited to accelerate the algorithm.   As a result, CP2 has quadratic convergence rate (i.e. the restricted primal dual gap after $N$ iterations has decay $O(1/N^2)$).  However, these advantages are offset by the fact that the convergence rate of CP2 depends heavily on the grid size.

In Tables~\ref{table:cp2_rof_10} and \ref{table:cp2_rof_20} we present the results of CP2 on the same experimental setup.     CP2 accelerates PDHG by changing the step sizes $\tau=\tau_n$ and $\sigma=\sigma_n$ at each iteration $n$ according to a special update rule.   One only needs to choose initial values $\tau_0$ and $\sigma_0$ satisfying $\tau_0\sigma_0\norm{\Delta_M}_{L^2}\leq 1$ where $\norm{\Delta_M}_{L^2}$ is the  $L^2$ norm of the discrete Laplace operator $\Delta_M$.    In \cite{chambolle_pock}, it is suggested to take $\tau_0$ extremely large.   We seemed to obtain the best results by choosing $\tau_0=\sigma_0=\frac{1}{\sqrt{\norm{\Delta_M}_{L^2}}}$, thus we report results with this choice.    Comparing Tables \ref{table:cp2_rof_10} and \ref{table:cp2_rof_20} to Tables \ref{table:rof_lambda_10} and \ref{table:rof_lambda_20}, we see that CP2 is slower in both time and iterations in every case.   In some of the cases, CP2 was unable to complete the computation in the alloted time (2 hours).

\begin{table}
\center{\caption{CP2 (from~\cite{chambolle_pock}) ROF disc $\lambda=10$}\label{table:cp2_rof_10}}

\begin{tabular}{ c | c | c | c | c |}
\cline{2-5}
 & \multicolumn{2}{|c|}{Error $10^{-2}$} & \multicolumn{2}{|c|}{Error $10^{-3}$} \\  \hline 
\multicolumn{1}{|c|}{Grid size} &  Iterations & Time (s) & Iterations & Time (s) \\ \hline
\multicolumn{1}{|c|}{$512 \times 512$} &  2196 & 15.6 & 4473 & 31.8 \\ \hline
\multicolumn{1}{|c|}{$1024 \times 1024$} & 4415  & 156.2 & 8471  & 296.2 \\ \hline
\multicolumn{1}{|c|}{$2048 \times 2048$} & 8855 &  1305.2 & 17724  &  2576.3\\ \hline
\multicolumn{1}{|c|}{$4096 \times 4096$} & ---  & --- & --- & --- \\ \hline
\end{tabular}
\end{table}

\begin{table}
\center{\caption{CP2 (from~\cite{chambolle_pock}) ROF disc $\lambda=20$}\label{table:cp2_rof_20}}

\begin{tabular}{ c | c | c | c | c |}
\cline{2-5}
 & \multicolumn{2}{|c|}{Error $10^{-2}$} & \multicolumn{2}{|c|}{Error $10^{-3}$} \\ \hline 
\multicolumn{1}{|c|}{Grid size} &  Iterations & Time (s) & Iterations & Time (s) \\ \hline
\multicolumn{1}{|c|}{$512 \times 512$} & 1252 & 8.7 & 2439 & 17.4 \\ \hline
\multicolumn{1}{|c|}{$1024 \times 1024$} & 2497 & 80.4 & 4063 & 132.4 \\ \hline
\multicolumn{1}{|c|}{$2048 \times 2048$} & 5005 & 711.0 & 8044 & 1133.8 \\ \hline
\multicolumn{1}{|c|}{$4096 \times 4096$} & 10103  & 6980.2 & --- & --- \\ \hline
\end{tabular}
\end{table}

\subsection{EMD}

 For the earth mover's distance, we will consider two different 2-dimensional problems.    In the first problem, $\rho^1$ and $\rho^0$ are both measures supported on a disc of radius $1/4$ where $\rho^1$ is centered at $(5/8, 5/8)$ and $\rho^0$ is centered at $(3/8, 3/8)$.  In the second problem, $\rho^1$ and $\rho^0$ are delta measures at the points $(5/8, 5/8)$ and $(3/8, 3/8)$ respectively.  In both cases, $\rho^1$ is the translation of $\rho^0$ by the vector $(1/4, 1/4)$. 
 
 When two measures differ by a translation, it is possible to determine a minimizer $m^*=u^*+\nabla \psi$ explicitly \cite{villanibook1}.  Suppose that $\rho^1$ is given by translating $\rho^0$ by the vector $v$.  Then given a continuous, vector-valued test function $p$ on $[0,1]^d$, there exists a minimizer $m^*$ such that
 \begin{equation}\label{eq:measure_translation} (m^*,p)= \int_0^1 \int_{[0,1]^d}v\cdot p(x-tv) \, d\rho^0(x) \, dt.\end{equation}
It then follows that the EMD distance must be equal to $\abs{v}$.    Thus, the EMD distance for  both experiments is $\frac{1}{\sqrt{8}}$.  Due to grid anisotropy, the solution on a discrete grid will be different, but it will approach $\frac{1}{\sqrt{8}}$ as the grid becomes finer.  

In the case of the two discs, the densities $\rho^1$ and $\rho^0$ are $L^{\infty}$ functions. From formula (\ref{eq:measure_translation}) we can deduce that $u^*$ must have bounded $L^2$ norm.  In fact, $\norm{m^*}^2_{L^2}=\norm{u^*}^2_{L^2}+\norm{\nabla \psi}_{L^2}^2$, and thus $\norm{u^*}_{L^2}\leq \norm{m^*}_{L^2}=8^{-1/4}$.   Thus, we do not need to estimate the gap $\delta_F(R)$ --- it is already zero when $R>8^{-1/4}$.  This suggests that we can simply choose step sizes $\tau=\sigma=1$.  The performance of G-prox PDHG on the disc experiment is presented in Table \ref{table:gprox_emd_discs}.  The convergence rate is clearly independent of the grid size for both error tolerances $10^{-3}$ and $10^{-4}$.

\begin{table}
\center{\caption{EMD discs}\label{table:gprox_emd_discs}}

\begin{tabular}{ c | c | c | c | c |}
\cline{2-5}
 & \multicolumn{2}{|c|}{Error $10^{-3}$} & \multicolumn{2}{|c|}{Error $10^{-4}$} \\ \hline 
\multicolumn{1}{|c|}{Grid size} &  Iterations & Time (s) & Iterations & Time (s) \\ \hline
\multicolumn{1}{|c|}{$512 \times 512$} &  64 & 1.7 & 163 & 3.6 \\ \hline
\multicolumn{1}{|c|}{$1024 \times 1024$} & 64  & 7.5 & 167  & 16.0 \\ \hline
\multicolumn{1}{|c|}{$2048 \times 2048$} & 64 &  30.6 & 168  &  67.4\\ \hline
\multicolumn{1}{|c|}{$4096 \times 4096$} & 65  & 126.4 & 168 & 294.6 \\ \hline
\end{tabular}
\end{table}

The case of two delta measures is different.  From formula (\ref{eq:measure_translation}) we see that $m^*$ is a singular measure which concentrates on a one-dimensional line segment.  We know that $\nabla \psi\in L^{q}$ for $q<2$.   Therefore, $u^*=m^*-\nabla \psi$ is also a singular measure and does not have finite $L^2$ norm.  As such, we will need to use Theorem \ref{thm:emd_convergence} to help choose the step sizes.  On a grid with $M$ points, we will take $\tau= \min\Big(\sqrt{\frac{1}{\epsilon \abs{\log\epsilon}}} \, ,\,  2M^{1/4} \Big)$, which again consists only of quantities that are easily calculated at the start of computation.  Note that here we have dropped the dimensionality constant in Theorem \ref{thm:emd_convergence}, and used the trivial estimate $\textrm{EMD}(\rho^1,\rho^0)\leq 1$ to get $\sqrt{\frac{1}{\epsilon \abs{\log\epsilon}}}$.  To get the second term $2M^{1/4}$ we first use the inequality $\norm{u^*_M}_{L^2}\leq \norm{m^*_M}_{L^2}$. Then we note that $m^*_M$ is approximately supported on a 1-dimensional line segment and thus,
$$\norm{m^*_M}_{L^2}\leq 2M^{1/4}\norm{m^*_M}_{L^1}\leq 2M^{1/4}.$$

In Table \ref{table:emd_delta}  we present the results of G-prox PDHG on the delta measure experiment.  When $\epsilon=10^{-2}$, $\sqrt{\frac{1}{\epsilon \abs{\log\epsilon}}}$ is smaller than  $2M^{1/4}$ for every grid size.   Therefore, the optimal step size $\tau$ is the same for every grid, and the number of iterations needed to reach the error cutoff is always 30.  When $\epsilon=10^{-3}$, $\sqrt{\frac{1}{\epsilon \abs{\log\epsilon}}}$ is larger than $2M^{1/4}$ on the $512 \times 512$ grid, approximately equal on the $1024 \times 1024$ grid and smaller on the $2048\times 2048$ and $4096 \times 4096$ grids.  As a result, the  $2048\times 2048$ and $4096 \times 4096$ grids have nearly identical iteration counts while the algorithm converges faster on the $512 \times 512$ and $1024 \times 1024$ grids.   When $\epsilon=10^{-4}$,  $2M^{1/4}$ is smaller than $\sqrt{\frac{1}{\epsilon \abs{\log\epsilon}}}$ on every tested grid size.  Therefore, the algorithm converges faster on the smaller grids.  Once again, if we did not allow $\tau$ to depend on the grid discretization we would get nearly identical iteration counts for each grid, but with slower convergence.

\begin{center}
\begin{table}
\center{\caption{EMD delta measures}\label{table:emd_delta}}

\begin{tabular}{ c | c | c | c | c | c | c |}
\cline{2-7}
 & \multicolumn{2}{|c|}{Error $10^{-2}$ }& \multicolumn{2}{|c|}{Error $10^{-3}$} & \multicolumn{2}{|c|}{Error $10^{-4}$} \\ \hline 
\multicolumn{1}{|c|}{Grid size} &  Iterations & Time (s) &  Iterations & Time (s) &  Iterations & Time (s) \\ \hline
\multicolumn{1}{|c|}{$512 \times 512$} & 30 & 1.2 & 56 & 1.6 & 121 & 2.8 \\ \hline
\multicolumn{1}{|c|}{$1024 \times 1024$} & 30 & 5.5 & 81 & 9.1 & 149 & 14.6 \\ \hline
\multicolumn{1}{|c|}{$2048 \times 2048$} & 30 & 22.8 & 98 & 45.6 & 185 & 75.3 \\ \hline
\multicolumn{1}{|c|}{$4096 \times 4096$} & 30  & 83.3 & 101  & 201.9 & 236  & 417.8 \\ \hline 
\end{tabular}
\end{table}
\end{center}

In \cite{Solomon} the authors solve the $L^1$ EMD problem with a completely equivalent ADMM method.  However, \cite{Solomon} does not consider how to choose optimal step sizes.  For singular measures this can lead to a significant slowdown.   In Table \ref{table:emd_delta_bad_steps}  we repeat the delta measure experiment with non-optimal step sizes $\tau=\sigma=1$  .    Comparing Tables \ref{table:emd_delta} and \ref{table:emd_delta_bad_steps} we see that the non-optimal step sizes lead to runtimes that are up to fifty times slower.   These results highlight the need for our careful theoretical analysis.

\begin{center}
\begin{table}
\center{\caption{EMD delta measures non-optimal step sizes}\label{table:emd_delta_bad_steps}}

\begin{tabular}{ c | c | c | c | c | c | c |}
\cline{2-7}
 & \multicolumn{2}{|c|}{Error $10^{-2}$ }& \multicolumn{2}{|c|}{Error $10^{-3}$} & \multicolumn{2}{|c|}{Error $10^{-4}$} \\ \hline 
\multicolumn{1}{|c|}{Grid size} &  Iterations & Time (s) &  Iterations & Time (s) &  Iterations & Time (s) \\ \hline
\multicolumn{1}{|c|}{$512 \times 512$} & 93 & 2.4 & 611 & 12.0 & 1864 & 35.3  \\ \hline
\multicolumn{1}{|c|}{$1024 \times 1024$} & 112 & 11.9 & 1055 & 90.7 & 3835  & 322.9 \\ \hline
\multicolumn{1}{|c|}{$2048 \times 2048$} & 128 & 58.2 & 1747 & 675.6 & 8413 & 2985.2 \\ \hline
\multicolumn{1}{|c|}{$4096 \times 4096$} & 137  & 253.3 &  2530  & 4149.5  & --  & -- \\ \hline 
\end{tabular}
\end{table}
\end{center}

Finally, let us note that the EMD functional is not strongly convex in $L^2$.  Without strong convexity,  it is not possible to use the accelerated algorithms from \cite{chambolle_pock}.  As a result, G-prox PDHG will be orders of magnitude faster than PDHG type algorithms which do not use preconditioning. We can verify this by comparing our algorithm to the state-of-the-art results for the EMD problem presented in \cite{Li2}.  In \cite{Li2}, the authors approach the EMD minimization problem 
$$\min_{\nabla \cdot m=\rho^1-\rho^0} \int_{[0,1]^d}\abs{m}$$
by converting it into a different unconstrained saddle point problem
$$\min_{m}\max_p \int_{[0,1]^d}\abs{m}+(\nabla \cdot m+\rho^0-\rho^1, p)$$
and searching for the saddle point using PDHG.  Since the divergence operator $\nabla \cdot$ is not preconditioned by PDHG, the convergence rate of the algorithm depends on the grid size. Due to this dependence, \cite{Li2} only considers grids of size $256\times 256$ and less, and requires parallelization for efficient computation.  Notably, our serial algorithm on grids of size $512\times 512$ appears to be faster than their parallel algorithm on grids of size $256\times 256$.

\section{Conclusion}\label{sec:discussion}

In this paper we have shown that G-prox PDHG, a variant of the PDHG algorithm, can be used to solve large scale optimization problems with a convergence rate independent of the grid size.   We have demonstrated our results both theoretically and numerically for two important optimization problems, the ROF denoising model and the earth mover's distance between probability measures.  Our method is able to solve these problems on grids as large as $4096\times 4096$ in a few minutes ---a benchmark which seems to be out of reach for previous approaches.  

In future works we hope to further extend our analysis and numerical results to other large scale problems of interest.  Furthermore, we hope to use our approach to more efficiently simulate the dynamics of stiff differential equations involving total variation.

\section{Acknowledgements}
The authors are grateful to Wilfrid Gangbo for helpful discussions. 
\appendix

\section{The domain $[0,1]^d$}
Functions $u:[0,1]^d\to\RR$ have a natural extension $\tilde{u}$ to the larger domain $[-1,1]^d$ via even reflections.  We can define $\tilde{u}$ explicitly by taking $\tilde{u}(x_1,\ldots, x_d)=u(\abs{x_1},\ldots, \abs{x_d})$.   $u$ takes the same value on opposite boundaries, thus we can glue opposite boundaries together and identify $[-1,1]^d$ with the torus $\TT^d$.  This allows us to perform Fourier analysis on $[0,1]^d$, as we can manipulate the Fourier series of $\tilde{u}$ and then restrict the result back to $[0,1]^d$.  Therefore, any Fourier multiplier type operator, such as convolution, can be defined on $[0,1]^d$ (these operators can also be defined in physical space through the translation invariance of $[-1,1]^d$). 

Other extensions to $[-1,1]^d$ are possible, however even reflections are most natural for our purposes.
Since $\tilde{u}$ is even on $[-1,1]^d$, its Fourier series expansion is a cosine series.  Assuming $\nabla \tilde{u}$ exists, it should have a sine series expansion.  As a result, $\nabla \tilde{u}\cdot n=0$ on the boundary of $[0,1]^d$.  Thus, we see that the extension by even reflections can be used to automatically solve the Poisson equation on $[0,1]^d$ with zero Neumann boundary conditions.   

\section{ROF proofs}
\begin{lemma}\label{prop:rof_convolution}
Suppose that $u:[0,1]^d\to[a,b]$ is a function of bounded variation. Let $G_{\delta}(z)=\delta^{-d}e^{-\pi (z/\delta)^2}$  be the Gaussian kernel and let $u=G_{\delta}*u$. Then we have the following inequalities:

\begin{equation} \norm{u_{\delta}-u}_{L^q}^q\leq \delta\, \left(\frac{d}{2\pi}\right)^{\frac{1}{2}} ( b-a )^{q-1}\norm{u}_{TV} \end{equation}

and

$$\norm{\nabla u_{\delta}}_{L^2}^2\leq \frac{1}{\delta}\sqrt{ 2\pi d^3}\Big(b-a \Big) \norm{u}_{TV}.$$

\end{lemma}

\begin{proof}
$$\norm{u_{\delta}-u}^q_{L^q}=\int_{[0,1]^d} \bigg|\int_{\RR^d}G(z)\big(u(x+\delta z)-u(x)\big) dz\bigg|^q dx.$$
Using Jensen's inequality and the fact that $u$ maps to the bounded interval $[a,b]$, we bound the above by
$$\leq (b-a)^{q-1} \int_{\RR^d} G(z)\int_{[0,1]^d} \abs{u(x+\delta z)-u(x)} dx dz.$$
Next we use the fact that BV functions satisfy a global Lipschitz property to get
$$\leq \delta\norm{u}_{TV}\int_{\RR^d}\abs{z}G(z)=\delta\norm{u}_{TV}\, A_{d-1}\int_0^{\infty} r^de^{-\pi r^2} dr$$
where $A_{d-1}$ is the surface area of the sphere $S^{d-1}$.  The first result follows from the inequality 
$$A_{d-1}\int_0^{\infty} r^de^{-\pi r^2} \leq \sqrt{\frac{d}{2\pi}}.$$

Now we turn to estimating the $H^1$ norm of $u_{\delta}$.  We have

$$\norm{\nabla u_{\delta}}_{L^2}^2=\int_{[0,1]^d}\bigg|\int_{\RR^d} \nabla G_{\delta}(z) u(x+z) dz\bigg|^2 dx=\frac{1}{\delta^2}\int_{[0,1]^d}\bigg|\int_{\RR^d} \nabla G(z) u(x+\delta z) dz\bigg|^2 dx. $$
Since $\int_{\RR^d} \nabla G(z) dz=0$, we may replace the above with 

$$\frac{1}{\delta^2}\int_{[0,1]^d}\bigg|\int_{\RR^d} \nabla G(z)\big(u(x+\delta z)-u(x)\big) dz\bigg|^2 dx.$$ 
We have the estimate $\norm{\nabla G}_{L^1}\leq  \sqrt{2\pi d}$, thus the above is

$$\leq \frac{\sqrt{2\pi d}}{\delta^2}(b-a)\int_{\RR^d}\abs{\nabla G(z)}\int_{[0,1]^d} \abs{u(x+\delta z)- u(x)} dx dz$$
Again applying the global Lipschitz property of BV functions we get
$$\leq \frac{\sqrt{2\pi d}}{\delta}(b-a)\norm{u}_{TV}\int_{\RR^d}\abs{z}\abs{\nabla G(z)}$$
Finally, $\int_{\RR^d}\abs{z}\abs{\nabla G(z)}=d$. 

\end{proof}

\section{EMD proofs}
\begin{lemma} Let $\rho^1, \rho^0$ be probability measures and suppose that $\psi$ solves the Poisson equation $\Delta \psi=\rho^1-\rho^0$ on $[0,1]^d$ with zero Neumann boundary conditions. Let $G_{\delta}$ be the Gaussian kernel with width $\delta > 0$ and $\nabla \psi_{\delta}=\nabla \psi*G_{\delta}$.  Then
$$\norm{\nabla \psi_{\delta} -\nabla \psi}_{L^1}\leq C_d \,  \delta\,\big(\abs{\log(\delta)}+1\big) \int_{[0,1]^d} \abs{\rho^1-\rho^0} $$
where $C_d$ is a constant which depends on the dimension only.  
\end{lemma}

\begin{proof}

Let us define a linear operator $T_{\delta}$  
$$T_{\delta}h=\nabla \Delta^{-1}\big(G_{\delta}*h -h\big).$$
The current proposition is equivalent to
$$ \norm{T_{\delta}(\rho^1-\rho^0)}_{L^1} \leq C_d\delta \,\big(\abs{\log(\delta)} + 1\big)\int_{[0,1]^d} \abs{\rho^1-\rho^0}.$$
Let $\rho^i_k=G_{1/k}*\rho^i$.  Then $\rho^i_k$ is a smooth $L^1$ function.  If we assume that $T_{\delta}$ is a bounded operator on $L^1$, we can use lower semi-continuity to obtain

$$\norm{T_{\delta}(\rho^1-\rho^0)}_{L^1}\leq \liminf_{k\to\infty} \norm{T_{\delta}(\rho^1_k-\rho^0_k)}_{L^1}\leq \norm{T_{\delta}}_{L^1}\liminf_{k\to\infty} \norm{\rho^1_k-\rho^0_k}_{L^1}.$$
Applying Jensen's inequality to the last term we have
$$\liminf_{k\to\infty} \norm{\rho^1_k-\rho^0_k}_{L^1}\leq \int_{[0,1]^d} \abs{\rho^1-\rho^0}$$
Thus, it is enough to show that for smooth functions $h$ the operator $T_{\delta}$ satisfies
$$\norm{T_{\delta}h}_{L^1}\leq C_d\, \delta\, \big(\abs{\log(\delta)}+1\big) \, \norm{h}_{L^1}.$$

For smooth $h$ we have
$$T_{\delta} h(x)=\int_{\RR^d}G(z)\,\Delta^{-1}\big(\nabla h(x+\delta z)-\nabla h(x)\big)\, dz=  \int_{\RR^d} G(z) \Delta^{-1}\int_{0}^{\delta} z^TD^2h(x+tz)  \,dt\,dz$$
All of the operators applied to $h$ commute, so we have
$$\norm{T_{\delta}h }_{L^1}\leq \int_{0}^{\delta} \norm{D^2\Delta^{-1} h*\tilde{G}_{t}}_{L^1} dt$$
where $\tilde{G}(z)=z^T G(z)$.    Suppose that $q(t)\in (1, 2]$ for each $t\in(0,\delta]$.  Then 
$$\int_{0}^{\delta} \norm{D^2\Delta^{-1} h*\tilde{G}_{t}}_{L^{1}}\, dt\leq \int_{0}^{\delta} \norm{D^2\Delta^{-1} h*\tilde{G}_{t}}_{L^{q(t)}}\, dt .$$
Now we use the fact that $D^2\Delta^{-1}$ is a bounded operator on $L^q$ for $q\in (1,\infty)$.  Moreover, for $q\in (1,2]$, we have the operator norm bound
$$\norm{D^2\Delta^{-1}}_{L^q\to L^q}\leq \frac{C_d''}{q-1}$$
where $C_d''$ is a constant depending on the dimension only \cite{Stein1970}.  Using the above bound and then Young's convolution inequality we get
$$\int_{0}^{\delta} \norm{D^2\Delta^{-1} h*\tilde{G}_{t}}_{L^{q(t)}}\, dt\leq \int_0^{\delta} \frac{C_d''}{q(t)-1}\norm{h*\tilde{G}_t}_{L^{q(t)}}\, dt\leq \norm{h}_{L^1} \int_0^{\delta} \frac{C_d''\norm{\tilde{G}_t}_{L^{q(t)}}}{q(t)-1}\, dt.$$
The $L^{q}$ norm of $\tilde{G}_t$ satisfies
$$\norm{\tilde{G}_t}_{L^q}\leq C'_d t^{d(1-q) / q}$$
for some new constant $C'_d$.  Now we shall make the choice $q(t)=1+\frac{1}{d\abs{\log(t)}}$.  This gives us 
$$ \norm{h}_{L^1} \int_0^{\delta} \frac{C_d''\norm{\tilde{G}_t}_{L^{q(t)}}}{q(t)-1}\, dt\leq C_d \norm{h}_{L^1}\int_{0}^{\delta}  \abs{\log(t)}\, dt$$
where $C_d$ is again a new constant.  The inequality $\int_0^{\delta} \abs{\log(t)}\, dt\leq \delta \, \abs{\log(\delta)}+\delta$ finishes the proof. 

\end{proof}

\begin{lemma}   Let  $G_{\delta}(z)=\delta^{-d}e^{-\pi (z/\delta)^2}$  be the Gaussian kernel.   Then there exists a minimizer $u^*$ of the EMD functional such that $u_{\delta}=G_{\delta}*u^*$ has finite $L^2$ norm bounded by 
$$\norm{u_{\delta}}_2\leq \bigg( (2\delta)^{(1-d)}\textrm{EMD}(\rho^1,\rho^0)\bigg)^{1/2}$$
\end{lemma}
\begin{proof}

Young's convolution inequality automatically gives $\norm{u_{\delta}}_{L^2}\leq \norm{G_{\delta}}_{L^2}\norm{u^*}_1=\delta^{-d/2}\norm{u^*}_{L^1}.$  However, this does not take into account the structure of the EMD problem, and better results are possible.  

Consider first the case where $\rho^1$ and $\rho^0$ are delta measures at locations $x_1, x_0\in (0,1)^d$ respectively.  It is then known (\cite{villanibook1}) that the solution $m^*=u^*+\nabla \psi$ is given by 

$$ (p,m^*)=(x_1-x_0)\cdot \int_0^1 p\big(t(x_1-x_0)+x_0\big) dt$$  
The simplicity of the solution allows us to express $\norm{G_{\delta}* m^*}_2^2$ explicitly in the Fourier domain.  This will allow us to bound $\norm{G_{\delta} * u^*}_2^2$ since $m^*=u^*+\nabla \psi$ and $u^*$ and $\nabla \psi$ are orthogonal in the $L^2$ inner product.
$$\norm{G_{\delta} * m^*}_2^2=\frac{\norm{x_1-x_0}_2^2}{2^{d-1}\pi^2}\sum_{n\in \ZZ^d} \frac{\abs{\sin\big(\pi  (n,x_1)\big)-\sin\big(\pi(n,x_0)\big)}^2 }{\abs{(n,x_1-x_0)}^2}e^{-\pi \delta^2 \abs{n}^2}. $$
The inner sum can be bounded by 

$$ 1+\int_{\RR^d} \frac{\abs{\sin(\pi  (\xi,x_1))-\sin(\pi (\xi,x_0))}^2 }{\abs{(\xi,x_1-x_0)}^2}e^{-\pi \delta^2 \norm{\xi}_2^2}$$
Using the inequality
$$\abs{\sin(\pi  (\xi,x_1))-\sin(\pi (\xi,x_0))}\leq \abs{\sin\big(\pi (\xi,x_1-x_0)\big)}$$
we may bound the integral by 

$$\int_{\RR^d} \frac{\sin\big(\pi (\xi, x_1-x_0)\big)^2}{\abs{(\xi, x_1-x_0)}^2}e^{-\pi \delta^2\norm{\xi}_2^2}.$$
By rotating, we may assume that $x_1-x_0$ is parallel to the first standard basis vector $\bm{e}_1$.  Thus, the integral simplifies to 

$$\int_{\RR^d} \frac{\sin\big(\pi \xi_1\norm{ x_1-x_0}_2\big)^2}{\xi_1^2\norm{x_1-x_0}_2^2}e^{-\pi \delta^2\norm{\xi}_2^2}\leq \delta^{1-d}\int_{\RR} \frac{\sin(\pi \xi_1\norm{x_1-x_0}_2)^2}{\xi_1^2\norm{x_1-x_0}_2^2} d\xi_1.$$

The integral on the right hand side can be computed explicitly, and has value $\frac{\pi^2}{2\norm{x_1-x_0}_2}$.  Thus, we may conclude that
$$\norm{G_{\delta} * u^*}^2_2\leq \norm{G_{\delta} * m^*}^2_2\leq (2\delta)^{1-d}\norm{x_1-x_0}_2 .$$

Next we extend our result to sums of delta measures.  Suppose that $\rho^1=\frac{1}{k}\sum_{j=1}^k \delta_{x_j}$ and $\rho^0=\frac{1}{k}\sum_{j=1}^k \delta_{y_j}$ for some (potentially repeated) points $x_1, \ldots, x_k, y_1, \ldots, y_k \in (0,1)^d$.  Then there exists a minimizer $m^*$ with the form

$$(p, m^*)= \frac{1}{k}\sum_{j=1}^k \big(x_j-y_{\pi(j)}\big) \cdot \int_{0}^1 p\big( y_{\pi(j)}+t(x_j-y_{\pi(j)})\big) dt  $$
where $\pi$ is a permutation of $\{1, \ldots, k\}$ which solves the assignment problem

$$\pi \in \argmin_{\sigma \in S_k} \frac{1}{k}\sum_{j=1}^k \norm{x_j-y_{\sigma(j)}}_2.$$
Using the triangle inequality and then Jensen's inequality,  we have 
$$\norm{G_{\delta} * m^*}_2\leq  \frac{1}{k}\sum_{j=1}^k \big((2\delta)^{1-d}\norm{x_j-y_{\pi(j)}}_2\big)^{1/2}\leq \bigg((2\delta)^{1-d}\textrm{EMD}(\rho^1, \rho^0)\bigg)^{1/2} .$$

Finally, we wish to extend this result to general probability measures $\rho^1, \rho^0$.   Let $P_k$ be the set of all probability measures of the form $\mu=\frac{1}{k}\sum_{j=1}^k \delta_{x_j}$ for any list of (potentially repeated) points $x_1, \ldots, x_k\in (0,1)^d$.  Sums of delta measures are dense in the space of probability measures with the EMD topology \cite{villanibook2}, therefore there exist sequences $\rho_k^1$ and $\rho^0_k$ such that $\rho^1_k, \rho_k^0\in P_k$,   $\textrm{EMD}(\rho_k^1,\rho^1)\to 0$ and $\textrm{EMD}(\rho_k^0,\rho^0)\to 0$.  Using these sequences, we can choose for each $k$
$$u_k\in \argmin_{\nabla \cdot u=0}\int_{[0,1]^d} \abs{u+\nabla \psi_k}$$
where $\psi_k$is  the solution of the Poisson equation $\Delta \psi_k=\rho_k^1-\rho_k^0$, with zero Neumann boundary conditions.

The solutions $u_k$ have finite total variation bounded by 
$$2\norm{\nabla \psi_k}_1\leq 2C_{d}\int_{[0,1]^d} \abs{\rho_k^1-\rho_k^0}\leq 2C_{d}$$ 
where $C_{d}$ is the operator norm of $\norm{\nabla \Delta^{-1}}_{L^1\to L^{\frac{d}{d-1},w}}$.  
Thus, by the Banach--Alaoglu theorem, there exists a subsequence $u_{k_n}$ and a vector-valued measure $\tilde{u}$ such that for every bounded continuous test function $p$ we have   
$$ \lim_{n\to\infty} (u_{k_n}-\tilde{u}, p)=0.$$
Without loss of generality, we shall assume that this property holds for the full sequence $u_k$.   Lower semi-continuity gives $\norm{G_{\delta}*\tilde{u}}_2\leq \liminf_{k\to\infty} \norm{G_{\delta}* u_k}_2$.     Thus, if we can show $\tilde{u}\in \argmin_{\nabla \cdot u=0}\int_{[0,1]^d} \abs{u+\nabla \psi}$ we will be done. 

To that end, we note that 
\begin{equation*}
\begin{split}
\liminf_{k\to\infty} \; \textrm{EMD}(\rho^1_k, \rho_k^0)= &\liminf_{k\to\infty} \; \int_{[0,1]^d} \abs{u_k+\nabla \psi_k}\\
\geq& \sup_{\norm{\varphi}_{\infty}=1}\liminf_{k\to\infty} \,(u_k+\nabla \psi_k, \varphi)=\int_{[0,1]^d} \abs{\tilde{u}+\nabla \psi}. 
\end{split}
\end{equation*}
Next by the triangle inequality, we have

$$\textrm{EMD}(\rho^1_k, \rho^0_k)\leq \textrm{EMD}(\rho^1, \rho^1_k) + \textrm{EMD}(\rho^0, \rho^0_k) + \textrm{EMD}(\rho^1, \rho^0)$$
thus $\limsup_{k\to\infty} \textrm{EMD}(\rho^1_k, \rho^0_k)\leq \textrm{EMD}(\rho^1, \rho^0)=\inf_{\nabla \cdot u=0} \int_{[0,1]^d} \abs{u+\nabla \psi} $.   Putting everything together, we get the chain of inequalities

$$\limsup_{k\to\infty}\; \textrm{EMD}(\rho^1_k, \rho^0_k)\leq \inf_{\nabla \cdot u=0} \int_{[0,1]^d} \abs{u+\nabla \psi}\leq  \int_{[0,1]^d} \abs{\tilde{u}+\nabla \psi}\leq \liminf_{k\to\infty} \; \textrm{EMD}(\rho^1_k, \rho_k^0)$$
which completes the proof.

\end{proof}

\bibliographystyle{siamplain}

\begin{thebibliography}{10}

\bibitem{caselles_chambolle_novaga_discontinuity}
{\sc V.~Caselles, A.~Chambolle, and M.~Novaga}, {\em The discontinuity set of
  solutions of the {TV} denoising problem and some extensions}, Multiscale
  Modeling \& Simulation, 6 (2007), pp.~879--894,
  \url{https://doi.org/10.1137/070683003}.

\bibitem{chambolle_pock}
{\sc A.~Chambolle and T.~Pock}, {\em A first-order primal-dual algorithm for
  convex problems with applications to imaging}, J. Math. Imaging Vision, 40
  (2011), pp.~120--145, \url{https://doi.org/10.1007/s10851-010-0251-1}.

\bibitem{chambolle_pock_imaging_book}
{\sc A.~Chambolle and T.~Pock}, {\em An introduction to continuous optimization
  for imaging}, Acta Numerica, 25 (2016), p.~161–319,
  \url{https://doi.org/10.1017/S096249291600009X}.

\bibitem{chambolle_pock_improved}
{\sc A.~Chambolle and T.~Pock}, {\em On the ergodic convergence rates of a
  first-order primal-dual algorithm}, Mathematical Programming, 159 (2016),
  pp.~253--287, 
  \url{https://doi.org/10.1007/s10107-015-0957-3}.

\bibitem{drs}
{\sc J.~Douglas and H.~H. Rachford}, {\em On the numerical solution of heat
  conduction problems in two and three space variables}, Transactions of the
  American Mathematical Society, 82 (1956), pp.~421--439,
\url{https://doi.org/10.1090/S0002-9947-1956-0084194-4}.

\bibitem{zeidler_optimization}
{\sc Z.~E.}, {\em Nonlinear Functional Analysis and its Applications. III.
  Variational Methods and Optimization}, Springer, New York, first~ed., 1985.

\bibitem{Evans}
{\sc L.~C. Evans and W.~Gangbo}, {\em Differential equations methods for the
  {M}onge--{K}antorovich mass transfer problem}, Memoirs of the American
  Mathematical Society, American Mathematical Society, Providence, RI, 1999,
  \url{http://cds.cern.ch/record/2122679}.

\bibitem{gabay_mercier}
{\sc D.~Gabay and B.~Mercier}, {\em A dual algorithm for the solution of
  nonlinear variational problems via finite element approximation}, Computers
  \& Mathematics with Applications, 2 (1976), pp.~17 -- 40,
  \url{https://doi.org/10.1016/0898-1221(76)90003-1}.

\bibitem{glowinski_1975}
{\sc R.~Glowinski and A.~Marroco}, {\em Sur l'approximation, par
  {\'e}l{\'e}ments finis d'ordre un, et la r{\'e}solution, par
  p{\'e}nalisation-dualit{\'e} d'une classe de probl{\`e}mes de {D}irichlet non
  lin{\'e}aires}, ESAIM: Mathematical Modelling and Numerical Analysis -
  Mod{\'e}lisation Math{\'e}matique et Analyse Num{\'e}rique, 9 (1975),
  pp.~41--76, \url{http://eudml.org/doc/193269}.

\bibitem{goldstein_osher}
{\sc T.~Goldstein and S.~Osher}, {\em The split {B}regman method for
  {$L1$}-regularized problems}, SIAM J. Imaging Sci., 2 (2009), pp.~323--343,
  \url{https://doi.org/10.1137/080725891}.

\bibitem{grafakos_classical_2014}
{\sc L.~Grafakos}, {\em Classical {F}ourier Analysis}, Graduate Texts in
  Mathematics, Springer New York, 2014,
  \url{https://books.google.com/books?id=94FxBQAAQBAJ}.

\bibitem{Li1}
{\sc W.~Li, E.~K. Ryu, S.~Osher, W.~Yin, and W.~Gangbo}, {\em A parallel method
  for earth mover's distance}, Journal of Scientific Computing,  (2017),
  pp.~1--16.

\bibitem{lions_mercier_79}
{\sc P.~L. Lions and B.~Mercier}, {\em Splitting algorithms for the sum of two
  nonlinear operators}, SIAM Journal on Numerical Analysis, 16 (1979),
  pp.~964--979, \url{https://doi.org/10.1137/0716071}.

\bibitem{nemirovski_saddle}
{\sc A.~Nemirovski}, {\em Prox-method with rate of convergence {$O(1/t)$} for
  variational inequalities with {L}ipschitz continuous monotone operators and
  smooth convex-concave saddle point problems}, SIAM Journal on Optimization,
  15 (2004), pp.~229--251, \url{https://doi.org/10.1137/S1052623403425629}.

\bibitem{nesterov_composite}
{\sc Y.~Nesterov}, {\em Gradient methods for minimizing composite functions},
  Math. Program., 140 (2013), pp.~125--161,
  \url{https://doi.org/10.1007/s10107-012-0629-5}.

\bibitem{nesterov_book}
{\sc Y.~Nesterov}, {\em Introductory Lectures on Convex Optimization: A Basic
  Course}, Applied Optimization, Springer US, 2013,
  \url{https://books.google.com/books?id=2-ElBQAAQBAJ}.

\bibitem{PeyreCuturi2018}
{\sc G.~Peyr{\'e} and M.~Cuturi}, {\em Computational optimal transport}, ArXiv
  e-prints,  (2018), \url{https://arxiv.org/abs/1803.00567}.

\bibitem{rof}
{\sc L.~I. Rudin, S.~Osher, and E.~Fatemi}, {\em Nonlinear total variation
  based noise removal algorithms}, Physica D: Nonlinear Phenomena, 60 (1992),
  pp.~259 -- 268,
  \url{https://doi.org/10.1016/0167-2789(92)90242-F}.

\bibitem{Li2}
{\sc E.~K. Ryu, W.~Li, P.~Yin, and S.~Osher}, {\em Unbalanced and partial
  {$L_1$} {M}onge\textendash{}{K}antorovich problem: A scalable parallel
  first-order method}, Journal of Scientific Computing,  (2017), pp.~1--18.

\bibitem{Solomon}
{\sc J.~Solomon, R.~Rustamov, L.~Guibas, and A.~Butscher}, {\em Earth mover's
  distances on discrete surfaces}, ACM Trans. Graph., 33 (2014),
  pp.~67:1--67:12, \url{https://doi.org/10.1145/2601097.2601175}.

\bibitem{Stein1970}
{\sc E.~M. Stein}, {\em Singular integrals and differentiability properties of
  functions}, Princeton Mathematical Series, No. 30, Princeton University
  Press, Princeton, N.J., 1970.

\bibitem{villanibook1}
{\sc C.~Villani}, {\em Topics in Optimal Transportation}, vol.~58 of Graduate
  Studies in Mathematics, American Mathematical Society, Providence, RI, 2003,
  \url{https://doi.org/10.1007/b12016}.

\bibitem{villanibook2}
{\sc C.~Villani}, {\em Optimal Transport: Old and New}, vol.~338 of Grundlehren
  der Mathematischen Wissenschaften, Springer Science \& Business Media, 2009,
  \url{https://doi.org/10.1007/978-3-540-71050-9}.

\bibitem{wang_lucier}
{\sc J.~Wang and B.~J. Lucier}, {\em Error bounds for finite-difference methods
  for {R}udin--{O}sher--{F}atemi image smoothing}, SIAM Journal on Numerical
  Analysis, 49 (2011), pp.~845--868, \url{https://doi.org/10.1137/090769594}.

\end{thebibliography}

\end{document}